\documentclass[a4paper,10pt,reqno]{amsart}

\usepackage{amssymb,amsmath,amsfonts,amsthm,mathtools}
\usepackage{latexsym,mathrsfs,pb-diagram}
\usepackage[pdftex]{graphicx}
\usepackage[all]{xy}

\usepackage[hmargin=2.5cm,vmargin=3.5cm]{geometry}
\usepackage[plainpages=false,colorlinks,pdfpagelabels]{hyperref}
\hypersetup{
  urlcolor=black,
  citecolor=green,
  linkcolor=blue
}

\numberwithin{equation}{section}

\theoremstyle{definition}
\newtheorem{Def}{Definition}[section]

\theoremstyle{remark}

\theoremstyle{plain}
\newtheorem{Prop}[Def]{Proposition}
\newtheorem{Cor}[Def]{Corollary}
\newtheorem{Thm}[Def]{Theorem}
\newtheorem{Lem}[Def]{Lemma}

\newcommand{\dfn}{\doteq}
\newcommand{\st}{ \ ; \ }
\newcommand{\lra}{\longrightarrow}
\newcommand{\sset}{\subset}

\newcommand{\Z}{\mathbb{Z}}

\newcommand{\R}{\mathbb{R}}

\newcommand{\C}{\mathbb{C}}

\newcommand{\TR}[5]{\begin{array}{c c c c c}
    {#1} & : & {#3} & \longrightarrow & {#5}\\
    & & {#2} & \longmapsto & {#4}
  \end{array}
}

\newcommand{\transp}[1]{\prescript{\mathrm{t} \!}{}{{#1}}}

\DeclareMathOperator{\ran}{\mathrm{ran}}

\newcommand{\del}{\partial}
\newcommand{\dd}{\mathrm{d}}
\renewcommand{\Re}{\mathsf{Re}}
\renewcommand{\Im}{\mathsf{Im}}

\newcommand{\D}{\mathscr{D}}
\newcommand{\cinfty}{\mathscr{C}^\infty}
\newcommand{\gev}{\mathscr{G}}
\newcommand{\sob}{\mathscr{H}}

\newcommand{\gr}{\mathfrak}
\DeclareMathOperator{\ad}{\mathrm{ad}}
\newcommand{\vv}{\mathrm}

\author{Gabriel Ara\'{u}jo}
\address{Universidade de S{\~a}o Paulo, ICMC-USP, S{\~a}o Carlos, SP, Brazil}
\email{\texttt{gccsa@icmc.usp.br}}

\author{Igor A.~Ferra}
\address{Universidade Federal do ABC, CMCC-UFABC, S{\~a}o Bernardo do Campo, SP, Brazil}
\email{\texttt{ferra.igor@ufabc.edu.br}}

\author{Luis F.~Ragognette}
\address{Universidade Federal de S{\~a}o Carlos, DM-UFSCar, S{\~a}o Carlos, SP, Brazil}
\email{\texttt{luisragognette@dm.ufscar.br}}

\thanks{This work was supported by the S{\~a}o Paulo Research Foundation (FAPESP, grants~2016/13620-5 and~2018/12273-5).}

\title[]{Global analytic hypoellipticity and solvability of certain operators subject to group actions}

\begin{document}

\begin{abstract}
  On $T \times G$,
  where $T$ is a compact real-analytic manifold
  and $G$ is a compact Lie group,
  we consider differential operators $P$
  which are invariant by left translations on $G$
  and are elliptic in $T$.
  Under a mild technical condition,
  we prove that
  global hypoellipticity of $P$
  implies its global analytic-hypoellipticity
  (actually Gevrey of any order $s \geq 1$).
  We also study the connection
  between the latter property
  and the notion of
  global analytic (resp.~Gevrey) solvability,
  but in a much more general setup.
\end{abstract}

\maketitle


\section*{Introduction}

Two notoriously difficult problems
in PDE theory are to determine
whether a general linear differential operator
$P$ defined, say, on a compact manifold $M$,
is globally hypoelliptic or globally solvable.
In very general terms,
given a reasonable space of functions
$\mathscr{F}$ on $M$,
the first problem means to determine
if we can infer from the information 
$Pu \in \mathscr{F}$ that
$u$ itself belongs to $\mathscr{F}$,
where $u$ is a priori taken in some
larger space of (generalized) functions.
The second one means to solve
the equation $Pu = f$
for ``every'' $f \in \mathscr{F}$,
with a solution $u$ also in $\mathscr{F}$
-- but a subtlety soon arises,
for this is generally impossible
as the global geometry of both $M$ and $P$
impose natural constraints on
the right-hand side $f$;
we compromise by asking instead
if we can always solve
at least for those ``admissible'' $f \in \mathscr{F}$.

In order to make the situation more manageable,
extra geometrical hypotheses may be imposed to
the problem.
A traditional one is to assume
$M$ endowed with a smooth action
of a Lie group $G$
which leaves $P$ invariant;
the action may be further assumed
transitive
(see e.g.~\cite[Chapter~5]{wallach_haohs}
and related works in the references therein)
or free,
in which case the operator
induced by $P$ on the orbit space $M/G$
plays a key role.

Here we are interested in
the latter possibility,
and deal with the simplest such situation:
$M$ is a product $T \times G$,
where $T$ is a compact manifold
and $G$ is a compact Lie group,
which acts freely on $M$
by left-multiplication on
the second factor alone
(where $G$ acts on itself).
Invariance of $P$ under this action
allows us to put it in a global canonical form
with separate variables~\eqref{eq:Pdefinition},
where the operator
$P_0$ induced by $P$ on the orbit space $T$
can be easily read off;
this one will be further required to be elliptic in $T$.
Operators $P$ satisfying
these properties essentially
are said to belong to class~$\mathcal{T}$
(Definition~\ref{def:classT})
and are our main focus.
When
$M = \mathbb T^n_t \times \mathbb T^m_x$,
where $\mathbb T^N$ is
the $N$-dimensional torus,
the invariance of $P$
under the action of $\mathbb T^m$
means that
$P = P(t,D_t,D_x)$
and there is a vast literature about
global hypoellipticity and solvability
of such operators in this ambient.
Requiring the ellipticity condition on $P_0$
imposes certain restrictions,
but still
there are important classes of such operators,
for instance sums of squares
of certain real vector fields
(see e.g.~\cite{chim94,hp00,albanese11}).

As for the remaining ingredient --
the space of functions $\mathscr{F}$
on which our questions
will be posed --
we are mainly interested in the space
of real-analytic functions
(for which we of course need to require
further regularity of all the objects involved):
global analytic hypoellipticity and solvability
of $P$ are then addressed.
Actually,
our approach applies to the
more general Gevrey classes
of functions $\gev^s$ of order $s \geq 1$
(see e.g.~\cite{rodino_gevrey})
of which real-analytic functions
are a special case ($s = 1$),
and one can even guess that
the program could be carried forward
for other ultradifferentiable classes
of Roumieu type
(even quasianalytic ones).

As far as the
regularity problem is concerned,
we consider the following question:
when does global hypoellipticity
(i.e.~in the smooth setup)
implies
global analytic hypoellipticity?
Positive answers
for this kind of question
first appeared in~\cite{chim94},
where the authors proved,
in contrast to
the famous example due to Baouendi and Goulaouic,
that the H\"ormander condition
(so local hypoellipticity
-- which is much stronger),
together with additional hypotheses
(that for us reads as the
ellipticity condition for
the class $\mathcal T$),
ensure global analytic hypoellipticity.
Similar results
in a more general framework
can be found in~\cite{christ94,brccj16}.
Our main result about this subject
is Theorem~\ref{thm:smooth_hyp_implies_gevrey},
which ensures that every operator
$P$ in class~$\mathcal{T}$
that is globally hypoelliptic
is also
globally analytic hypoelliptic
(actually globally $\gev^s$-hypoelliptic
for every $s \geq 1$).
Theorem~\ref{thm:smooth_hyp_implies_gevrey}
is more properly related
to~\cite[Theorem~2.2]{hp06}
and~\cite[Theorem~1.10]{petronilho09}
(and indeed a kind of generalization of them)
which deal with the case when
both $T$ and $G$ are tori.
Their results
were extended
for more general classes
of functions in~\cite{aj14}. 

The proof of
Theorem~\ref{thm:smooth_hyp_implies_gevrey}
is obtained by first
retrieving global estimates
from microlocal information
(Section~\ref{sec:microlocal})
and then combining them
with some results
(Section~\ref{sec:gevrey_vectors})
related
to the notion of Gevrey vectors
of the partial Laplace-Beltrami operator $\Delta_G$
(associated with a suitable metric on $G$),
a business of interest in PDE on its own.


In Section~\ref{sec:solvability}
we tackle the issue of solvability
from a far more abstract viewpoint.
We forget the existence of symmetries
and prove (Theorem~\ref{thm:gevrey_agh})
that for a general operator
$P$ on a compact manifold $M$
satisfying a regularity property
much weaker than
global $\gev^s$-hypoellipticity
we have that the map
$P: \gev^s(M) \to \gev^s(M)$
has closed range --
which, we argue,
is the correct notion
of global $\gev^s$-solvability
(an analogous result in the smooth category
has been recently proved to be true~\cite{afr20b}).
This result follows from a theory
of regularity for abstract operators
acting on certain pairs of
topological vector spaces
started in~\cite{araujo18}
whose development we continue here.
The reader will notice that,
thanks to the generality of our approach,
the proofs here apply equally well to
many other classes of operators:
certain pseudodifferential operators
acting on $\gev^s(M)$
(not necessarily of finite order),
and so on;
either scalar or vector-valued.

Finally,
we turn back to the situation
with symmetries
linking our class $\mathcal T$
with the general property addressed
in the abstract results of Section~\ref{sec:solvability}:
a converse of Theorem~\ref{thm:gevrey_agh}
is proved for operators in that class
(Proposition~\ref{prop:converse_agh_classT}).

\section{Preliminaries}
\label{sec:prelim}

Let $M$ be a real-analytic manifold,
assumed throughout to be
compact, connected and oriented.
The space $\gev^s(M)$ of globally defined
Gevrey functions~\cite{rodino_gevrey}
can be characterized
by means of the powers of a suitable
real-analytic elliptic operator~\cite{kn62,bcm79}:
here, we endow $M$ with a real-analytic Riemannian metric
(which is always possible~\cite{grauert58})
and denote by $\Delta_M$
the underlying Laplace-Beltrami operator
acting on functions,
in which case a smooth $f$ belongs to $\gev^s(M)$
if and only if there exist $C, h > 0$ such that
\begin{equation*}
  \| (I + \Delta_M)^k f \|_{L^2(M)} \leq C h^k k!^{2s},
  \quad \forall k \in \Z_+.
\end{equation*}
The case $s = 1$ describes
the space of real-analytic functions.
As in~\cite{araujo18},
we furnish topologies
to these spaces as follows:
for each $h > 0$ we let
\begin{equation*}
  \gev^{s, h}(M) \dfn  \Big\{ f \in \cinfty(M) \st \sup_{k\in \Z_+} h^{-k} k!^{-2s} \| (I + \Delta_M)^k f \|_{L^2(M)} < \infty \Big\}
\end{equation*}
which is a Banach space;
as the parameter $h$ increases,
these are contained into
one another in a continuous and compact fashion,
meaning that their union $\gev^s(M)$,
now endowed with the injective limit topology,
is what one calls a DFS space.

We denote by $\sigma(\Delta_M)$
the spectrum of $\Delta_M$.
By ellipticity of $\Delta_M$,
its $\lambda$-eigenspace $E_\lambda^M$ is
a finite dimensional subspace of $\gev^1(M)$;
connectedness of $M$ ensures that $E_0^M = \C$.
Denoting by 
$\mathcal{F}_\lambda^M: L^2(M) \to E_\lambda^M$
the orthogonal projection
we have\
\begin{equation}
  \label{eq:full_fourier_series}
  f = \sum_{\lambda \in \sigma(\Delta_M)} \mathcal{F}_\lambda^M(f)
\end{equation}
with convergence in $L^2(M)$;
this is an abstract analog of
Fourier series.
As such,
we extend it to Schwartz distributions:
given $f \in \D'(M)$
-- which we identify with a continuous linear functional
on $\cinfty(M)$ using the underlying volume form --
we let $\mathcal{F}_\lambda^M(f)$ be the unique element
in $E_\lambda^M$ such that
\begin{equation*}
  \langle \mathcal{F}_\lambda^M(f), \phi \rangle_{L^2(M)} = \langle f, \bar{\phi} \rangle,
  \quad \forall \phi \in E_\lambda^M.
\end{equation*}
In this situation~\eqref{eq:full_fourier_series} still holds,
but with convergence in $\D'(M)$;
when $f$ is smooth this convergence takes place in $\cinfty(M)$,
and so on.
Thanks to Weyl's asymptotic estimates~\cite[p.~155]{chavel_eigenvalues},
a distribution $f$ 
belongs to $\gev^s(M)$ if and only
if
there exist $C, h > 0$ such that
\begin{equation*}
  \| \mathcal{F}_\lambda^M(f) \|_{L^2(M)} \leq C e^{-h (1 + \lambda)^{\frac{1}{2s}}},
  \quad \forall \lambda \in \sigma(\Delta_M).
\end{equation*}
This allows us to consider
alternatively the adapted norms
\begin{equation}
  \label{eq:tildeGevreynorm}
  \| f \|_{\tilde{\gev}^{s,h}(M)}
  \dfn \Big( \sum_{\lambda \in \sigma(\Delta_M)} e^{2h(1 + \lambda)^{\frac{1}{2s}}}  \| \mathcal{F}^M_\lambda (f) \|_{L^2(M)}^2 \Big)^{\frac{1}{2}}
\end{equation}
and the spaces
\begin{equation}
  \label{eq:tildeGevreyspace}
  \tilde{\gev}^{s,h}(M) \dfn \{ f \in \cinfty(M) \st \| f \|_{\tilde{\gev}^{s,h}(M)} < \infty \}
\end{equation}
which are better suited
to some applications
(e.g.~in the proof of
Lemma~\ref{lem:global_ineq_microl_hormander_gevrey});
as $h > 0$ increases,
this produces new
Banach spaces compactly contained
in one another and
forming a directed system equivalent
to the previous one,
hence with the same
injective limit $\gev^s(M)$
-- both set-theoretically and topologically.
On time,
it is also convenient
to introduce the following
adapted Sobolev norms:
\begin{equation}
  \label{eq:adaptedSobolev}
  \| f \|_{\sob^t(M)}
  \dfn \Big( \sum_{\lambda \in \sigma(\Delta_M)} (1 + \lambda)^{2t} \| \mathcal{F}_\lambda^M (f) \|_{L^2(M)}^2 \Big)^{\frac{1}{2}},
  \quad t > 0.
\end{equation}
  
When $M$ is a product $T \times G$ of two such manifolds
and carrying the product metric,
an abstract theory of \emph{partial} Fourier series
can also be developed
(for details see e.g.~\cite{afr20, afr20b}).
In that case one proves that
$\Delta_M = \Delta_T + \Delta_G$
as differential operators on $M$,
that any $\alpha \in \sigma(\Delta_M)$ is of the form
$\alpha = \mu + \lambda$ for some
$\mu \in \sigma(\Delta_T)$ and $\lambda \in \sigma(\Delta_G)$
(and vice versa), and
\begin{equation*}
  E_\alpha^M
  = \bigoplus_{\substack{\mu \in \sigma(\Delta_T) \\ \lambda \in \sigma(\Delta_G) \\ \mu + \lambda = \alpha}} E_\mu^T \otimes E_\lambda^G.
\end{equation*}
Moreover,
given $\mu \in \sigma(\Delta_T)$ and $\lambda \in \sigma(\Delta_G)$
we will fix, whenever necessary,
bases for $E_\mu^T$ and $E_\lambda^G$
\begin{align*}
  \{ \psi^\mu_i \st 1 \leq i \leq d^T_\mu \},
  &\quad \text{where $d^T_\mu \dfn \dim E^T_\mu$}, \\
  \{ \phi^\lambda_j \st 1 \leq j \leq d^G_\lambda \},
  &\quad \text{where $d^G_\lambda \dfn \dim E^G_\lambda$},
\end{align*}
which are orthonormal w.r.t.~the inner products
inherited from $L^2(T), L^2(G)$, respectively,
in which case
\begin{equation*}
  \mathcal{S} \dfn \{ \psi^\mu_i \otimes \phi^\lambda_j \st 1 \leq i \leq d^T_\mu, \ 1 \leq j \leq d^G_\lambda, \ \mu \in \sigma(\Delta_T), \ \lambda \in \sigma(\Delta_G) \}
\end{equation*}
is a Hilbert basis for $L^2(T \times G)$.
Now given $f \in \D'(T \times G)$
and $\lambda \in \sigma(\Delta_G)$
one defines an object
$\mathcal F_\lambda^G (f) \in \D'(T;E_\lambda^G) \cong \D'(T)\otimes E_\lambda^G$
that in terms of our choice of basis
can be concretely written as
\begin{equation*}
  \mathcal{F}_\lambda^G (f) = \sum_{j=1}^{d_\lambda^G} \mathcal{F}_\lambda^G(f)_j \otimes \phi_j^\lambda,
\end{equation*}
where $\mathcal F_\lambda^G(f)_j \in \D'(T)$ is defined by
\begin{equation*}
  \langle \mathcal F_\lambda^G(f)_j, \psi \rangle \dfn \langle f, \psi \otimes \overline{\phi^\lambda_j} \rangle,
  \quad \forall \psi \in \cinfty(T).
\end{equation*}
The ``total'' Fourier projection of $f$
can then be recovered from the partial ones as
\begin{equation*}
  \mathcal{F}^{M}_\alpha(f) = \sum_{\mu + \lambda=\alpha} \mathcal{F}^{T}_\mu \mathcal{F}^G_\lambda(f),
  \quad \alpha \in \sigma(\Delta_M).
\end{equation*}

Gevrey functions can also be described
by the ``partial'' Fourier projections:
\begin{Prop}
  \label{Pro:partial-convergence-Gs}
  If $f \in \gev^s (T \times G)$
  then
  $\mathcal{F}_\lambda^G (f) \in \gev^s (T; E_\lambda^G)$
  for every
  $\lambda$
  and
  \begin{equation*}
    f = \sum_{\lambda \in \sigma (\Delta_G)} \mathcal{F}_\lambda^G (f)  
  \end{equation*}
  with convergence in
  $\gev^s(T \times G)$.
\end{Prop}

\section{A class of invariant operators}

Let $T$ be a real-analytic
Riemannian manifold,
assumed compact, connected and oriented,
and let $G$ be a compact,
connected Lie group,
carrying a Riemannian metric
which is $\ad$-invariant~\cite[Proposition~4.24]{knapp_lgbi}.
We denote by $\gr{g}$
the Lie algebra associated with $G$.
Let $P$ be a LPDO on $T \times G$
that is invariant
by the left action of $G$,
i.e.,
if $L_g: T \times G \lra T \times G$
is defined by
$L_g(t, x) \dfn (t, gx)$
then for every
$u \in \cinfty(T \times G)$ we have
\begin{equation}
  \label{eq:invariance}
  (L_g)^*(P u) = P[(L_g)^* u],
  \quad \forall g \in G.
\end{equation}
We call such operators
\emph{$G$-invariant}.
By choosing a basis $\vv{X}_1, \ldots, \vv{X}_m \in \gr{g}$
-- which we regard as a global frame for the tangent space of $G$ --
we may write $P$ as follows:
\begin{equation}
  \label{eq:Pdefinition}
  P = \sum_{|\alpha|  \leq r} P_{\alpha} \vv{X}^{\alpha}
\end{equation}
where $P_{\alpha}$ is a LPDO on $T$
and
$\vv{X}^{\alpha} = \vv{X}_{1}^{\alpha_1} \ldots \vv{X}_{m}^{\alpha_m}$
for each multi-index $\alpha = (\alpha_1, \ldots, \alpha_m) \in \Z_+^m$.

Indeed,
let $U \sset T$ and $V \sset G$
be coordinate open sets, i.e.,
there exist $\tilde{U} \subset \R^{n}$ and $\tilde{V} \subset \R^{m}$
and real-analytic charts
\begin{equation*}
  \phi: U \lra \tilde{U},
  \quad
  \psi: V \lra \tilde{V}.
\end{equation*}
We can write,
in $U \times V$,
\begin{equation*}
  P = \sum_{|\alpha|+|\beta| \leq r} a^{U,V}_{\alpha,\beta} \ \partial_t^\beta \vv{X}^\alpha,
  \quad a^{U,V}_{\alpha,\beta} \in \cinfty (U \times V)
\end{equation*}
and first we are going to prove that
if $\psi_1: V_1 \lra \tilde{V}_1$
and $\psi_2: V_2 \lra \tilde{V}_2$
are two charts with $V_1\cap V_2 \neq \varnothing$
then
\begin{equation}
  \label{eq:face-operator-eq1}
  a^{U,V_1}_{\alpha,\beta} = a^{U,V_2}_{\alpha,\beta}
  \quad \text{in $U \times (V_1 \cap V_2)$}
\end{equation}
for every $|\alpha|+|\beta| \leq r$.
It is enough to prove that
\begin{equation*}
  a_{\alpha,\beta}^{U,V_1} \circ (\phi \times \psi_1)^{-1}
  = a_{\alpha,\beta}^{U,V_2} \circ (\phi \times \psi_1)^{-1}
  \quad \text{in $\tilde{U} \times \psi_1 (V_1 \cap V_2)$}.
\end{equation*}
If we denote
the transition map 
\begin{equation*}
  \xi
  \dfn (\phi \times \psi_2) \circ (\phi \times \psi_1)^{-1}
  = \mathrm{Id}_{\tilde{U}} \times (\psi_2 \circ \psi_1^{-1})
  : \tilde{U} \times \psi_1(V_1 \cap V_2) \lra \tilde{U} \times \psi_2(V_1 \cap V_2)
\end{equation*}
then the pullback
$[ (\phi \times \psi_1)^{-1} ]^* P$
of $P$ to the open set
$\tilde{U} \times \psi_1 (V_1 \cap V_2)$
is given by\footnote{Notice the ambiguity in the notation:
  we are denoting by
  $\del_{t_1}, \ldots, \del_{t_n}$
  both the standard Euclidean partial derivatives
  on $\tilde{U} \sset \R^n$
  and their pullbacks to $U$ via $\phi$.}
\begin{align*}
  [ (\phi \times \psi_1)^{-1}] ^* P
  &= \xi^* [ (\phi \times \psi_2)^{-1} ]^*
  \left( \sum_{|\alpha| + |\beta| \leq r} a_{\alpha,\beta}^{U,V_2} \ \partial_{t}^\beta \vv{X}^\alpha \right) \\
  &= \xi^*
  \left( \sum_{|\alpha| + |\beta| \leq r} a_{\alpha,\beta}^{U,V_2} \circ (\phi \times \psi_2)^{-1} \ \partial_{t}^\beta \left(\psi_2^{-1}\right)^* \vv{X}^\alpha\right) \\
  &= \sum_{|\alpha| + |\beta| \leq r} a_{\alpha,\beta}^{U,V_2} \circ (\phi \times \psi_2)^{-1} \circ \xi \ \partial_{t}^\beta (\psi_2 \circ \psi_1^{-1})^* (\psi_2^{-1})^* \vv{X}^\alpha \\
  &= \sum_{|\alpha| + |\beta| \leq r} a_{\alpha,\beta}^{U,V_2} \circ (\phi \times \psi_2)^{-1} \circ \xi \ \partial_{t}^\beta (\psi_1^{-1})^* \vv{X}^\alpha.
\end{align*}
On the other hand
\begin{align*}
  [ (\phi \times \psi_1)^{-1} ]^* P
  &= [ (\phi \times \psi_1)^{-1}]^*
  \left( \sum_{|\alpha| + |\beta| \leq r} a^{U,V_1}_{\alpha,\beta} \ \partial_t^\beta \vv{X}^\alpha \right) \\
  &= \sum_{|\alpha| + |\beta| \leq r} a^{U,V_1}_{\alpha,\beta} \circ (\phi \circ \psi_1)^{-1} \ \partial_t^\beta (\psi_1^{-1})^* \vv{X}^\alpha.
\end{align*}
Since
$\{ \phi^* \partial_{t_1}, \ldots, \phi^* \partial_{t_n}, \vv{X}_1, \ldots, \vv{X}_m \}$
is a frame in
$U \times (V_1 \cap V_2)$
then 
\begin{equation*}
  \{ \partial_{t_1}, \ldots, \partial_{t_n}, (\psi_1^{-1})^* \vv{X}_1, \ldots, (\psi_1^{-1})^* \vv{X}_m \}
\end{equation*}
is a frame in
$\tilde{U} \times \psi_1 (V_1 \cap V_2)$,
hence by~\cite[Theorem~1.1.2]{varadarajan84}
we conclude that
\begin{equation*}
  a_{\alpha,\beta}^{U,V_1} \circ (\phi \times \psi_2)^{-1} \circ \xi
  = a^{U,V_1}_{\alpha,\beta} \circ (\phi \times \psi_1)^{-1}
\end{equation*}
for every
$|\alpha| + |\beta| \leq r$,
which in turn gives~\eqref{eq:face-operator-eq1}. 
Hence for a given coordinate chart
$(U, \phi)$ in $T$ as above,
if we set
$a^{U}_{\alpha,\beta}(t,x) = a^{U,V}_{\alpha,\beta}(t,x)$
for $x$ in the domain of some coordinate chart $(V, \psi)$
then $a^{U}_{\alpha,\beta} \in \cinfty(U \times G)$
and we can write
\begin{equation*}
  P = \sum_{|\alpha| + |\beta| \leq r} a^{U}_{\alpha,\beta} \ \partial_t^\beta \vv{X}^\alpha
  \quad \text{in $U \times G$}.
\end{equation*}

The next step is to prove that
$a^{U}_{\alpha,\beta}$ does not depend
on the $x$-variable.
We employ~\eqref{eq:invariance};
notice that both 
$\partial_t^\beta$ and $\vv{X}^\alpha$
are $G$-invariant as well
-- the latter by left-invariance of
$\vv{X}_1, \ldots, \vv{X}_m$.
For $u \in \cinfty_c(U \times G)$
we have
\begin{align*}
  \sum_{|\alpha| + |\beta|\leq r} a^{U}_{\alpha,\beta}(t,gx) \partial_t^\beta \vv{X}^\alpha u(t,gx)
  &= (L_g)^* (Pu)(t,x) \\
  &= (P[(L_g)^*u])(t,x) \\
  &= \sum_{|\alpha| + |\beta| \leq r} a^{U}_{\alpha,\beta}(t,x) \partial_t^\beta \vv{X}^\alpha [(L_g)^*u] (t,x) \\
  &= \sum_{|\alpha| + |\beta| \leq r} a^{U}_{\alpha,\beta}(t,x) [(L_g)^* (\partial_t^\beta \vv{X}^\alpha u)] (t,x) \\
  &= \sum_{|\alpha| + |\beta| \leq r} a^{U}_{\alpha,\beta}(t,x) \partial_t^\beta \vv{X}^\alpha u (t, gx).
\end{align*}
Since $u$ is arbitrary
this shows that, as LPDOs in $U \times G$,
\begin{equation*}
  \sum_{|\alpha| + |\beta| \leq r} a^{U}_{\alpha,\beta}(t,gx) \ \partial_t^\beta \vv{X}^\alpha
  = \sum_{|\alpha| + |\beta| \leq r} a^{U}_{\alpha,\beta}(t,x) \ \partial_t^\beta \vv{X}^\alpha,
  \quad \forall g \in G.
\end{equation*}
Thanks to~\cite[Theorem~1.1.2]{varadarajan84}
again we conclude that
$a^{U}_{\alpha,\beta}(t,gx) = a^{U}_{\alpha,\beta}(t,x)$
for every
$t \in U$ and $x, g \in G$,
which proves that $a^{U}_{\alpha,\beta}$ does not depend on $x$. 

In short,
we have proved that
given a coordinate chart
$(U, \phi)$ in $T$,
there exist functions
$a^U_{\alpha,\beta} \in \cinfty (U)$
such that
\begin{equation*}
  P = \sum_{|\alpha| + |\beta| \leq r} a^U_{\alpha,\beta} \ \partial_t^\beta \vv{X}^\alpha
  \quad \text{in $U \times G$}.
\end{equation*}
Thus if we write
$P_{U, \alpha} \dfn \sum_{|\beta| \leq r-|\alpha|} a^U_{\alpha,\beta} \partial_t^\beta$
then $P_{U, \alpha}$ is a LPDO in $U$,
\begin{equation*}
  P = \sum_{|\alpha| \leq r} P_{U, \alpha} \vv{X}^\alpha
  \quad \text{in $U \times G$}
\end{equation*}
and it remains to prove that
$P_{U,\alpha}$ is in fact globally defined,
that is if
$\phi_1: U_1 \lra \tilde{U}_1$
and
$\phi_2: U_2 \lra \tilde{U}_2$
are two charts in $T$ then
$P_{U_1, \alpha} = P_{U_2, \alpha}$ in $U_1 \cap U_2$.
The latter equality means that
\begin{equation*}
  P_{U_1, \alpha} (v) = P_{U_2, \alpha}(v),
  \quad \forall v \in \cinfty_c(U_1 \cap U_2).
\end{equation*}
If we fix
$v \in \cinfty_c(U_1 \cap U_2)$
and $t_0 \in U_1 \cap U_2$ then
\begin{equation*}
  Q_1 \doteq \sum_{|\alpha| \leq r} [P_{U_1, \alpha}(v)(t_0)] \ \vv{X}^\alpha
  \quad \mbox{and} \quad
  Q_2 \doteq \sum_{|\alpha| \leq r} [P_{U_2, \alpha}(v)(t_0)] \ \vv{X}^\alpha
\end{equation*}
are LPDOs in $G$
and satisfy
\begin{equation*}
  Q_1(h)(x) = P(v \otimes h)(t_0,x) = Q_2(h)(x),
  \quad \forall h \in \cinfty(G),
\end{equation*}
so $Q_1 = Q_2$.
Then we can use
one more time~\cite[Theorem~1.1.2]{varadarajan84}
to conclude that
$P_{U_1, \alpha}(v)(t_0) = P_{U_2,\alpha}(v)(t_0)$
for every $|\alpha| \leq r$;
since $t_0 \in U_1 \cap U_2$
and $v \in \cinfty_c(U_1 \cap U_2)$
are arbitrary
our proof is complete.

It is then clear that
$[P, \Delta_G] = 0$
(since the $P_\alpha$ are operators in $T$ only
and all the vector fields
$\vv{X}_1, \ldots, \vv{X}_m$
commute with $\Delta_G$
-- thanks to the general remark that
left-invariant vector fields on $G$ always commute
with the Laplace-Beltrami operator associated to
an $\ad$-invariant metric). 
As such,
$P$ acts as an endomorphism of both
$\cinfty(T; E_\lambda^G)$ and $\D'(T; E_\lambda^G)$
for each $\lambda \in \sigma(\Delta_G)$,
which induces by restriction a differential operator
$\widehat{P}_\lambda$ on $T \times E_\lambda^G$,
the latter regarded as a trivial vector bundle over $T$.
These operators are expressed as follows:
given $\psi \in \cinfty(T; E_\lambda^G)$,
written as
\begin{equation*}
  \psi = \sum_{i = 1}^{d_\lambda^G} \psi_i \otimes \phi_i^\lambda,
  \quad \psi_i \in \cinfty(T),
\end{equation*}
we have
\begin{equation*}
  \widehat{P}_\lambda \psi
  =  \sum_{i = 1}^{d_\lambda^G}(P_{0} \psi_i) \otimes \phi_i^\lambda
  + \sum_{i = 1}^{d_\lambda^G} \sum_{0<|\alpha|\leq r}(P_{\alpha} \psi_i) \otimes (\vv{X}^{\alpha} \phi_i^\lambda)
\end{equation*}
which we put on canonical form recalling that
\begin{equation*}
  \vv{X}^{\alpha} \phi_i^\lambda = \sum_{j = 1}^{d_\lambda^G} \gamma_{ij}^{\lambda \alpha} \phi_j^\lambda,
  \quad \gamma_{ij}^{\lambda \alpha} \in \C,
\end{equation*}
hence
\begin{equation}
  \label{eq:localPlambda}
  \widehat{P}_\lambda \psi
  = \sum_{j = 1}^{d_\lambda^G} \left( P_0\psi_j
  + \sum_{i = 1}^{d_\lambda^G} \sum_{0<|\alpha|  \leq r}  \gamma_{ij}^{\lambda \alpha} P_\alpha \psi_i\right) \otimes \phi_j^\lambda.
\end{equation}

From here on
we will assume that $P$ and $P_0$
have the same order $r$.
Thus the order of $P_{\alpha}$ is at most $ r-|\alpha|$,
and expression~\eqref{eq:localPlambda}
reveals that $\widehat{P}_\lambda$,
as a differential operator on the vector bundle
$T \times E_\lambda^G$ over $T$,
also has order $r$:
its principal part
is essentially the same as that of $P_0$.
Actually,
under such circumstances 
the following identity between
their principal symbols holds:
\begin{equation}
  \label{eq:relationship_symbols}
  \mathrm{Symb}_{(t_0, \tau_0)} (\widehat{P}_\lambda)
  = \mathrm{Symb}_{(t_0, \tau_0)} (P_0) \cdot \mathrm{id}_{E^G_\lambda},
  \quad \forall (t_0, \tau_0) \in T^* T \setminus 0.
\end{equation}

Indeed, recall that
the principal symbol of $\widehat{P}_\lambda$
may be regarded as a linear map
\begin{align*}
  \mathrm{Symb}_{(t_0, \tau_0)} (\widehat{P}_\lambda): E_\lambda^G \lra E_\lambda^G
\end{align*}
which can be computed as follows.
Given $\psi \in \cinfty(T; \R)$ such that $\dd \psi(t_0) = \tau_0$
we have, for every $\phi \in E_\lambda^G$:
\begin{equation*}
  \mathrm{Symb}_{(t_0, \tau_0)} (\widehat{P}_\lambda) \phi
  = \lim_{\rho \to \infty} \rho^{-r} \left. e^{-i \rho \psi} \widehat{P}_\lambda(e^{i \rho \psi} (1_T \otimes \phi)) \right|_{t_0}
  = \lim_{\rho \to \infty} \rho^{-r} \left. e^{-i \rho \psi} \widehat{P}_\lambda(e^{i \rho \psi} \otimes \phi) \right|_{t_0}
\end{equation*}
Here,
$1_T \otimes \phi$ plays the role of
a section of $T \times E^G_\lambda$
whose value at $t_0$ is $\phi$
and $r$ is the order of $\widehat{P}_\lambda$.
But
\begin{equation*}
  \widehat{P}_\lambda(e^{i \rho \psi} \otimes \phi)
  = P_0 (e^{i \rho \psi}) \otimes \phi
  + \sum_{0< |\alpha| \leq r} P_{\alpha} (e^{i \rho \psi}) \otimes \vv{X}^{\alpha} \phi;
\end{equation*}
note now that when $|\alpha| > 0$
we have that
$e^{-i\rho\psi}P_{\alpha} (e^{i \rho \psi})$
is a polynomial in $\rho$
of degree strictly less then $r$, 
and since the order of $P_0$ is $r$ we have
\begin{align*}
  \mathrm{Symb}_{(t_0, \tau_0)} (\widehat{P}_\lambda) \phi
  &= \left( \lim_{\rho \to \infty} \rho^{-r}  e^{-i \rho \psi} P_0 (e^{i \rho \psi}) \Big|_{t_0} \right) \phi
  + \sum_{0<|\alpha| \leq r }  \left( \lim_{\rho \to \infty} \rho^{-r}   e^{-i\rho\psi}P_{\alpha} (e^{i \rho \psi})\Big|_{t_0} \right) \vv{X}^{\alpha} \phi \\
  &= ( \mathrm{Symb}_{(t_0, \tau_0)} (P_0)) \phi
\end{align*}
thus proving~\eqref{eq:relationship_symbols}.

In particular:
\begin{Prop}
  \label{prop:ell_plambda}
  Suppose that
  $P$ as in~\eqref{eq:Pdefinition}
  has the property that
  $P$ and $P_0$ have the same order.
  Then the following are equivalent:
  \begin{enumerate}
  \item $P_0$ is elliptic in $T$.
  \item $\widehat{P}_\lambda$ is elliptic in $T \times E^G_\lambda$
    for some $\lambda \in \sigma(\Delta_G)$.
  \item $\widehat{P}_\lambda$ is elliptic in $T \times E^G_\lambda$
    for every $\lambda \in \sigma(\Delta_G)$.
  \end{enumerate}
\end{Prop}
\begin{proof}
  Ellipticity of $\widehat{P}_\lambda$ at $t_0 \in T$
  is characterized by injectivity of
  $\mathrm{Symb}_{(t_0, \tau)} (\widehat{P}_\lambda)$
  for every $\tau \in T_{t_0}^* T \setminus 0$,
  which by~\eqref{eq:relationship_symbols}
  is equivalent to
  \begin{equation*}
    \mathrm{Symb}_{(t_0, \tau)} (P_0) \neq 0,
    \quad \forall \tau \in T^*_{t_0} T \setminus 0
  \end{equation*}
  i.e.~to $P_0$ being elliptic at $t_0$.
\end{proof}
This motivates us
to introduce the following class
of operators on $T \times G$.
\begin{Def}
  \label{def:classT}
  We say that a LPDO $P$ on $T \times G$
  belongs to \emph{class~$\mathcal{T}$} if
  \begin{enumerate}
  \item $P$ is $G$-invariant,
  \item $P$ and $P_0$ in~\eqref{eq:Pdefinition} have the same order and
  \item $P_0$ is elliptic in $T$.
  \end{enumerate}
\end{Def}

\subsection{Examples}

\subsubsection{Vector fields}
\label{exa:invariant_vf}

Consider $\vv{Y}$ a vector field
on $T \times G$ that can be  written as
\begin{equation}
  \label{eq:invariant_vf}
  \vv{Y} = \vv{W} + \sum_{j = 1}^m a_j(t) \vv{X}_j
\end{equation}
where
$a_1, \ldots, a_m \in \cinfty(T)$
and $\vv{W} \in \gr{X}(T)$
is a complex vector field.
Notice that $\vv{W}$
has the same order of $\vv{Y}$
unless it vanishes identically.
Therefore,
$\vv{Y}$ belongs to $\mathcal{T}$
if and only if
$\vv{W}$ is elliptic on $T$.

When $\vv{W}$ is a real vector field
the latter condition forces $T$ to be one-dimensional
i.e.~$T = S^1$,
in which case $\vv{W}$ is a non-vanishing
multiple of $\del_t$,
and $\vv{Y}$ is equivalent to
a vector field of the form
\begin{equation*}
  \del_t + \sum_{j = 1}^m a_j(t) \vv{X}_j
  \quad \text{on $S^1 \times G$}.
\end{equation*}
The case $\dim T = 2$ is also possible
but in that case $\vv{W}$ must be complex
with $\Re \vv{W}$ and $\Im \vv{W}$
linearly independent everywhere;
if we additionally assume that these
real vector fields commute
then one can endow $T$ with the structure
of a Riemann surface such that $\vv{W}$
is essentially the Cauchy-Riemann operator
i.e.~$\vv{Y}$ is of the form
\begin{equation*}
  \del_{\bar{z}} + \sum_{j = 1}^m a_j(t) \vv{X}_j
  \quad \text{on $T \times G$}.
\end{equation*}
When $\dim T \geq 3$
no vector field $\vv{Y}$ on $T \times G$
will belong to $\mathcal{T}$
as in that case $\vv{W}$ can never be elliptic.

\subsubsection{Certain second-order operators}

Let $Q$ be a second-order operator on $T$ and
\begin{equation*}
  \vv{Y}_\ell \dfn \vv{W}_\ell + \sum_{j = 1}^m a_{\ell j}(t) \vv{X}_j,
  \quad \ell \in \{1, \ldots, N\},
\end{equation*}
be vector fields
of the form~\eqref{eq:invariant_vf}.
Then
\begin{equation*}
  P \dfn Q - \sum_{\ell = 1}^N \vv{Y}_\ell^2
\end{equation*}
is $G$-invariant,
and moreover belongs to $\mathcal{T}$
if and only if
\begin{equation*}
  Q - \sum_{\ell = 1}^N \vv{W}_\ell^2
\end{equation*}
is elliptic of order $2$ on $T$.
Real operators in this class were investigated
e.g.~when $Q = \Delta_T$~\cite{afr20, afr20b};
or when $Q = 0$ and $\vv{W}_1, \ldots, \vv{W}_N$ span
the tangent bundle of $T$ everywhere~\cite{brccj16}
(further references therein).

\section{From microlocal analysis to global Gevrey estimates}
\label{sec:microlocal}

\begin{Lem}
  \label{lem:micr_from_fio_gevrey}
  Let $P$ be a real-analytic LPDO on $T \times G$
  belonging to class~$\mathcal{T}$.
  If $u \in \D'(T \times G)$ is such that
  $Pu \in \gev^s(T \times G)$
  then for every $\phi \in \gev^s(G)$
  we have that
  $\tilde{u}(\phi) \dfn \langle u, \cdot \otimes \phi \rangle \in \gev^s(T)$.
\end{Lem}
\begin{proof}
  Fix $t \in T$
  and notice that
  $\mathrm{WF}_s(u)$ does
  intercept the conormal bundle
  of $\{ t \} \times G$.
  Indeed,
  a covector
  $(t, \tau, x, \xi) \in T^* (T \times G)$
  annihilates $T_{(t,x)} (\{ t \} \times G)$
  if and only if $\xi = 0$,
  and none of these belongs to
  the characteristic set of $P$:
  as one can easily compute,
  \begin{equation*}
    \mathrm{Symb}_{(t, \tau, x, 0)} (P) = \mathrm{Symb}_{(t, \tau)} (P_0)
  \end{equation*}
  and $P_0$
  is elliptic in $T$ by assumption.
  Our claim follows since
  $\mathrm{WF}_s(u) \sset \mathrm{Char}(P)$~\cite[Theorem~5.1]{hormander71b}.
  We are then allowed to
  restrict $u$ to $\{t \} \times G$,
  i.e.~to pull it back 
  via the map
  $x \in G \mapsto (t,x) \in T \times G$,
  yielding in this way
  a distribution $u_t \in \D'(G)$,
  and by~\cite[Theorem~4.1]{hormander71b}
  the function
  $t \in T \mapsto \langle u_t, \phi \rangle \in \C$
  belongs to $\gev^s(T)$
  whatever $\phi \in \gev^s(G)$.
  That function is,
  however,
  none other than
  the distribution
  $\psi \in \cinfty(T) \mapsto \langle u, \psi \otimes \phi \rangle \in \C$.
\end{proof}

\begin{Cor}
  \label{cor:ulambdagevrey}
  Under the hypotheses of Lemma~\ref{lem:micr_from_fio_gevrey}
  we have $\mathcal{F}^G_\lambda (u) \in \gev^s(T; E_\lambda^G)$
  for every $\lambda \in \sigma(\Delta_G)$.
\end{Cor}
\begin{proof}
  Apply Lemma~\ref{lem:micr_from_fio_gevrey}
  to a basis of $E_\lambda^G$.
\end{proof}

For the next result,
given $\theta \in (0, 1)$ we define the set
\begin{equation*}
  \Gamma_\theta \dfn \{ (\mu, \lambda) \in \sigma(\Delta_T) \times \sigma(\Delta_G) \st \lambda \leq \theta \mu \}. 
\end{equation*}
\begin{Lem}
  \label{lem:global_ineq_microl_hormander_gevrey}
  Suppose that $u \in \D'(T \times G)$ is such that
  $\tilde{u}(\phi) \in \gev^s(T)$
  for every $\phi \in \gev^s(G)$.
  Then there exist $C, h > 0$
  and $\theta \in (0, 1)$
  such that
  \begin{equation}
    \label{eq:gevrey_cone_estimate}
    \| \mathcal{F}^T_\mu \mathcal{F}^G_\lambda (u) \|_{L^2(T \times G)} \leq C e^{-h (1 + \mu + \lambda)^{\frac{1}{2s}}},
    \quad \forall (\mu, \lambda) \in \Gamma_\theta. 
  \end{equation}
\end{Lem}
\begin{proof}
  We employ the DFS space
  characterization
  of the Gevrey spaces as
  in~\eqref{eq:tildeGevreynorm}-\eqref{eq:tildeGevreyspace}.
  By hypothesis we have
  $\tilde{u}(\tilde{\gev}^{s,1}(G)) \sset \gev^s(T)$
  hence the induced linear map
  $\tilde{u}: \tilde{\gev}^{s,1}(G) \to \gev^s(T)$
  is continuous by
  De Wilde's Closed Graph Theorem~\cite[p.~57]{kothe_tvs2}
  (whose applicability is granted by the fact
  that Banach spaces are ultrabornological
  and DFS spaces are webbed):
  indeed, notice that its graph is closed
  thanks to the continuity of
  $\tilde{u}: \cinfty(G) \to \D'(T)$.
  As such, $\tilde{u}$ maps
  bounded sets in $\tilde{\gev}^{s,1}(G)$ to
  bounded sets in $\gev^{s}(T)$
  so by~\cite[Lemma~3]{komatsu67} we may assert
  the existence of an $h > 0$ such that
  \begin{equation*}
    \tilde{u} ( \{ \phi \in \tilde{\gev}^{s,1}(G) \st \| \phi \|_{\tilde{\gev}^{s,1}(G)} \leq 1 \}) \sset \tilde{\gev}^{s,h}(T).
  \end{equation*}
  Of course, we may assume that $h > 2$.
  By linearity,
  $\tilde{u}(\tilde{\gev}^{s,1}(G)) \sset \tilde{\gev}^{s,h}(T)$
  and again
  $\tilde{u}: \tilde{\gev}^{s,1}(G) \to \tilde{\gev}^{s,h}(T)$
  is continuous by the Closed Graph Theorem
  (the classical one).
  Therefore,
  there exists a constant
  $C > 0$ such that
  \begin{equation*}
    \| \tilde{u}(\phi) \|_{\tilde{\gev}^{s,h}(T)} \leq C \| \phi \|_{\tilde{\gev}^{s,1}(G)},
    \quad \forall \phi \in \tilde{\gev}^{s,1}(G).
  \end{equation*}

  When we take
  $\phi = \overline{\phi^\lambda_j}$
  -- an element of our orthonormal basis of $E^G_\lambda$ --
  we obtain
  \begin{equation*}
    \| \tilde{u}(\overline{\phi^\lambda_j}) \|_{\tilde{\gev}^{s,h}(T)}^2
    = \sum_{\mu \in \sigma(\Delta_T)} e^{2h(1 + \mu)^{\frac{1}{2s}}}  \| \mathcal{F}^T_\mu [\tilde{u}(\overline{\phi^\lambda_j})] \|_{L^2(T)}^2 
    = \sum_{\mu \in \sigma(\Delta_T)} e^{2h(1 + \mu)^{\frac{1}{2s}}} \sum_{i = 1}^{d^T_\mu} | \langle u, \overline{\psi^\mu_i \otimes \phi^\lambda_j} \rangle |^2
  \end{equation*}
  hence
  \begin{equation*}
    \sum_{j = 1}^{d^G_\lambda} \| \tilde{u}(\overline{\phi^\lambda_j}) \|_{\tilde{\gev}^{s,h}(T)}^2
    = \sum_{\mu \in \sigma(\Delta_T)} e^{2h(1 + \mu)^{\frac{1}{2s}}} \sum_{i = 1}^{d^T_\mu} \sum_{j = 1}^{d^G_\lambda} | \langle u, \overline{\psi^\mu_i \otimes \phi^\lambda_j} \rangle |^2
    = \sum_{\mu \in \sigma(\Delta_T)} e^{2h(1 + \mu)^{\frac{1}{2s}}} \| \mathcal{F}^T_\mu \mathcal{F}^G_\lambda (u) \|_{L^2(T \times G)}^2
  \end{equation*}
  from which we conclude that
  \begin{equation*}
    e^{2h(1 + \mu)^{\frac{1}{2s}}} \| \mathcal{F}^T_\mu \mathcal{F}^G_\lambda (u) \|_{L^2(T \times G)}^2
    \leq \sum_{j = 1}^{d^G_\lambda} \| \tilde{u}(\overline{\phi^\lambda_j}) \|_{\tilde{\gev}^{s,h}(T)}^2
    \leq d^G_\lambda C^2 e^{2(1 + \lambda)^{\frac{1}{2s}}}
  \end{equation*}
 where we used that
  \begin{equation*}
    \| \overline{\phi^\lambda_j} \|_{\tilde{\gev}^{s, 1}(G)} = e^{(1 + \lambda)^{\frac{1}{2s}}}.
  \end{equation*}
  By Weyl's asymptotic formula $d_{\lambda}^{G}= O(\lambda^{m/2})$
  we have, enlarging $C$ if necessary,
  \begin{equation*}
    e^{2h(1 + \mu)^{\frac{1}{2s}}} \| \mathcal{F}^T_\mu \mathcal{F}^G_\lambda (u) \|^{2}_{L^2(T \times G)} \leq  C^{2} e^{4(1 + \lambda)^{\frac{1}{2s}}},
    \quad \forall (\mu, \lambda) \in \sigma(\Delta_T) \times \sigma(\Delta_G).
  \end{equation*}

  For $(\mu, \lambda) \in \Gamma_\theta$ we then have
  \begin{equation*}
    1 + \lambda \leq 1 + \theta \mu  \leq 1 + \mu
  \end{equation*}
  so
  \begin{equation*}
    \| \mathcal{F}^T_\mu \mathcal{F}^G_\lambda (u) \|_{L^2(T \times G)}
    \leq  C e^{2(1 + \lambda)^{\frac{1}{2s}} - h(1 + \mu)^{\frac{1}{2s}}}
    \leq  C e^{(2 - h)(1 + \mu)^{\frac{1}{2s}}}
  \end{equation*}
  but also
  \begin{equation*}
    1 + \mu + \lambda \leq 1 + \mu + \theta \mu \leq 2 (1 + \mu)
  \end{equation*}
  hence
  \begin{equation*}
    \| \mathcal{F}^T_\mu \mathcal{F}^G_\lambda (u) \|_{L^2(T \times G)} \leq C e^{-h'(1 + \mu + \lambda)^{\frac{1}{2s}}}
  \end{equation*}
  where $h' \dfn (h-2) / (2^{\frac{1}{2s}})$.
\end{proof}

\begin{Prop}
  \label{prop:partial_smoothness_converse}
  If $u \in \D'(T \times G)$ is such that
  \begin{enumerate}
  \item there exist $C, h > 0$ and $\theta \in (0,1)$
    such that~\eqref{eq:gevrey_cone_estimate} holds and
  \item there exist $C', h' > 0$ such that
    \begin{equation}
      \label{eq:gevrey_vector_exp}
      \| \mathcal{F}_\lambda^G (u) \|_{L^2(T \times G)} \leq C' e^{-h' (1 + \lambda)^{\frac{1}{2s}}},
      \quad \forall \lambda \in \sigma(\Delta_G)
    \end{equation}
  \end{enumerate}
  then $u \in \gev^s(T \times G)$.
\end{Prop}
\begin{proof}
  For $(\mu, \lambda) \in \sigma(\Delta_T) \times \sigma(\Delta_G)$
  not in $\Gamma_\theta$ we have
  \begin{align*}
    1 + \mu + \lambda < 1 + \frac{\lambda}{\theta} + \lambda \leq \frac{2}{\theta}(1 + \lambda)
  \end{align*}
  since $1/\theta > 1$, hence
  \begin{equation*}
    (1 + \lambda)^{\frac{1}{2s}} \geq  \left(\frac{\theta}{2} \right)^{\frac{1}{2s}} (1 + \mu + \lambda)^{\frac{1}{2s}}.
  \end{equation*}
  We conclude that
  \begin{equation*}
    \| \mathcal{F}^T_\mu \mathcal{F}^G_\lambda (u) \|_{L^2(T \times G)} \leq
    \begin{cases}
      C' e^{-h' (\theta/2)^{\frac{1}{2s}} (1 + \mu + \lambda)^{\frac{1}{2s}}}, &\text{in $\Gamma^c_\theta$} \\ 
      C e^{-h (1 + \mu + \lambda)^{\frac{1}{2s}}}, &\text{in $\Gamma_\theta$}
    \end{cases}
  \end{equation*}    
  and therefore
  $u \in \gev^s(T \times G)$.
\end{proof}

\section{Gevrey vectors for the partial Laplacian}
\label{sec:gevrey_vectors}

\begin{Prop}
  Let $P$ be a real-analytic LPDO on $T \times G$
  that is globally hypoelliptic and 
  commutes with $\Delta_G$.
  If $u \in \cinfty(T \times G)$ is such that
  $Pu \in \gev^s(T \times G)$ then
  $u$ is a \emph{$s$-Gevrey vector} for $\Delta_G$,
  that is,
  there exist $C, h > 0$ such that
  \begin{equation}
    \label{eq:gevrey_vector}
    \| \Delta_G^k u \|_{L^2(T \times G)} \leq C h^k k!^{2s}
  \end{equation}
  for every $k \in \Z_+$.
\end{Prop}
\begin{proof}
  We follow an approach similar to that of~\cite{brccj16},
  for which we start with some preliminary remarks.
  Let $u \in \cinfty(T \times G)$ be such that
  $Pu \in \gev^s(T \times G)$.
  Since $P$ is globally hypoelliptic,
  by standard functional analytic arguments
  (see e.g.~\cite[Lemma~3.1]{bmz12})
  there exists $t \in \R$ and $C_1 > 0$ such that
  \begin{equation*}
    \| \Delta_G^{k} u \|_{L^2(T \times G)}
    \leq C_1 ( \| P \Delta_G^{k} u \|_{\sob^t(T \times G)} + \| \Delta_G^{k} u \|_{\sob^{-1}(T \times G)}),
    \quad \forall k \in \Z_+.
  \end{equation*}
  We make use of
  the Sobolev norms~\eqref{eq:adaptedSobolev}.
  Then the last term
  in the inequality above
  can be estimated as follows:
  \begin{align*}
    \| \Delta_G^{k + 1} u \|_{\sob^{-1}(T \times G)}^2
    = \sum_{\substack{\mu \in \sigma(\Delta_T) \\ \lambda \in \sigma(\Delta_G)}}
    \frac{\lambda^2}{(1 + \mu + \lambda)^2} \| \mathcal{F}_\mu^T \mathcal{F}_\lambda^G (\Delta_G^{k}u) \|_{L^2(T \times G)}^2 
    \leq \| \Delta_G^{k} u \|_{L^2(T \times G)}^2,
    \quad \forall k \in \Z_+.
  \end{align*}
  Moreover, as $P$ commutes with $\Delta_G$ we have
  \begin{equation*}
    \| P \Delta_G^k u \|_{\sob^t(T \times G)}^2
    = \| \Delta_G^k P u \|_{\sob^t(T \times G)}^2
    = \sum_{\substack{\mu \in \sigma(\Delta_T) \\ \lambda \in \sigma(\Delta_G)}}
    (1 + \mu + \lambda)^{2t} \lambda^{2k} \| \mathcal{F}_\mu^T \mathcal{F}_\lambda^G (Pu) \|_{L^2(T \times G)}^2
  \end{equation*}
  and since $Pu \in \gev^s(T \times G)$
  there exists $A > 0$
  such that
  \begin{equation*}
    \| \mathcal{F}_\mu^T \mathcal{F}_\lambda^G (Pu) \|_{L^2(T \times G)}
    \leq A^{\ell + 1} \ell!^{2s} (1 + \mu + \lambda)^{-\ell},
    \quad \forall \ell \in \Z_+, \ \mu \in \sigma(\Delta_T), \ \lambda \in \sigma(\Delta_G)
  \end{equation*}
  hence
  \begin{equation}
    \label{eq:estimatell}
    \| P \Delta_G^k u \|_{\sob^t(T \times G)}^2
    \leq (A^{\ell + 1} \ell!^{2s})^2
    \sum_{\substack{\mu \in \sigma(\Delta_T) \\ \lambda \in \sigma(\Delta_G)}}
    (1 + \mu + \lambda)^{2t - 2 \ell} \lambda^{2k}.
  \end{equation}

  We choose $\ell_0 \in \Z_+$ such that
  \begin{equation*}
    B \dfn
   \bigg( \sum_{\substack{\mu \in \sigma(\Delta_T) \\ \lambda \in \sigma(\Delta_G)}}
    (1 + \mu + \lambda)^{2t - 2 \ell_0}\bigg)^{1/2}
    < \infty
  \end{equation*}
  which always exists thanks to
  Weyl's asymptotic formula,
  so plugging $\ell = \ell_0 + k$
  into~\eqref{eq:estimatell} yields:
  \begin{align*}
    \| P \Delta_G^k u \|_{\sob^t(T \times G)}^2
    &\leq (A^{\ell_0 + k + 1} (\ell_0 + k)!^{2s})^2
    \sum_{\substack{\mu \in \sigma(\Delta_T) \\ \lambda \in \sigma(\Delta_G)}}
    (1 + \mu + \lambda)^{2t - 2 \ell_0} \frac{\lambda^{2k}}{(1 + \mu + \lambda)^{2k}} \\
    &\leq B^2 (A^{\ell_0 + k + 1} (\ell_0 + k)!^{2s})^2
  \end{align*}
  and in particular we have that
  \begin{align}
    \label{eq:last_before_induction}
    \| \Delta_G^{k + 1} u \|_{L^2(T \times G)}
    &\leq C_1 (B A^{\ell_0 + k + 2} (\ell_0 + k + 1)!^{2s} + \| \Delta_G^k u \|_{L^2(T \times G)}) \nonumber \\
    &\leq C_1 (B A^{\ell_0 + k + 2} 4^{s(\ell_0 + k + 1)} \ell_0!^{2s} (k + 1)!^{2s} + \| \Delta_G^k u \|_{L^2(T \times G)})
  \end{align}
  for every $k \in \Z_+$.
  
  We are in position to assume
  by induction that there exist $C, h > 0$
  such that~\eqref{eq:gevrey_vector} holds up to a certain
  $k_0 \in \Z_+$: we will check that
  it also holds for $k = k_0 + 1$.
  We may assume w.l.o.g.~that
  \begin{equation}
    \label{eq:assumptions_Ch}
    C >  2 C_1 B A^{\ell_0 + 1} 4^{s \ell_0} \ell_0!^{2s},
    \quad
    h > \max \{ 2C_1, A 4^s \}.
  \end{equation}
  By~\eqref{eq:last_before_induction} we have
  \begin{equation*}
    \| \Delta_G^{k_0 + 1} u \|_{L^2(T \times G)}
    \leq C_1 (B A^{\ell_0 + k_0 + 2} 4^{s(\ell_0 + k_0 + 1)} \ell_0!^{2s} (k_0 + 1)!^{2s} + Ch^{k_0} k_0!^{2s})
  \end{equation*}
  where
  \begin{equation*}
    \frac{C_1 (B A^{\ell_0 + k_0 + 2} 4^{s(\ell_0 + k_0 + 1)} \ell_0!^{2s} (k_0 + 1)!^{2s} + C h^{k_0} k_0!^{2s})}{C h^{k_0 + 1} (k_0 + 1)!^{2s}} 
    = \frac{C_1 B A^{\ell_0 + 1} 4^{s\ell_0} \ell_0!^{2s}}{C} \left( \frac{A 4^s}{h} \right)^{k_0 + 1} + \frac{C_1}{ h (k_0 + 1)^{2s}}
  \end{equation*}
  is less than $1$
  thanks to~\eqref{eq:assumptions_Ch}.
  This concludes our proof.
\end{proof}

Notice that if
$u \in \cinfty(T \times G)$ is a
$s$-Gevrey vector for $\Delta_G$
then 
\begin{equation*}
  \lambda^k \| \mathcal{F}_\lambda^G(u) \|_{L^2(T \times G)}
  = \| \mathcal{F}_\lambda^G( \Delta_G^k u) \|_{L^2(T \times G)}
  \leq \| \Delta_G^k u \|_{L^2(T \times G)}
  \leq C h^k k!^{2s}
\end{equation*}
for some $C, h > 0$ independent
of both $\lambda \in \sigma(\Delta_G)$
and $k \in \Z_+$. Assuming that $h>1$ we obtain
\begin{equation*}
 (1+ \lambda)^k \| \mathcal{F}_\lambda^G(u) \|_{L^2(T \times G)}  \leq C  (2h)^k k!^{2s},
\end{equation*}
which, by standard computations,
is equivalent to~\eqref{eq:gevrey_vector_exp}
for some constants $C', h' > 0$.
Summing up with the results
from Section~\ref{sec:microlocal}
we immediately get:
\begin{Thm}
  \label{thm:smooth_hyp_implies_gevrey}
  Let $P$ be a real-analytic LPDO on $T \times G$
  belonging to class~$\mathcal{T}$.
  If $P$ is globally hypoelliptic in $T \times G$
  then for each $s \geq 1$
  it is also globally $\gev^s$-hypoelliptic:
  \begin{equation*}
    \forall u \in \D'(T \times G), \ Pu \in \gev^s(T \times G) \Longrightarrow u \in \gev^s(T \times G).
  \end{equation*}
\end{Thm}

\section{A general result and an application on global Gevrey solvability}
\label{sec:solvability}

Our goal in this section
is to introduce an abstract notion
of hypoellipticity of an operator
and show that this notion implies
closedness of range of the same operator.
As a consequence of this relationship
we prove that a very weak notion
of Gevrey hypoellipticity
is sufficient to global solvability
in the corresponding setup.

We take a look at the
\emph{category of pairs} of topological vector spaces:
its objects are $2$-tuples of topological vector spaces
$(E^\sharp, E)$,
where $E$ is a linear subspace of $E^\sharp$ carrying a  topology finer
  than that inherited from $E^\sharp$,
while its morphisms are \emph{maps of pairs}
$T: (E^\sharp, E) \to (F^\sharp, F)$,
meaning that $T: E^\sharp \to F^\sharp$ is a  continuous linear map
such that $T(E) \sset F$
and the induced map
$T: E \to F$
is continuous.
In this context:
\begin{Def}
  We shall say that $T$ satisfies:
  \begin{itemize}
  \item \emph{property~$(\mathcal{H})$} if
    for every $u \in E^\sharp$
    such that $Tu \in F$
    we have that $u \in E$;
  \item \emph{property~$(\mathcal{H}')$} if
    for every $u \in E^\sharp$
    such that $Tu \in F$
    there exists $v \in E$
    such that $Tv = Tu$.
  \end{itemize}
\end{Def}

Clearly 
$(\mathcal{H})$ holds 
if and only if
$T$ satisfies both~$(\mathcal{H}')$
and $\ker T \sset E$.
Moreover,
one has that
$(E^\sharp / \ker T, E / (E \cap \ker T))$
is also a pair of topological vector spaces
and $T$ descends to the quotient as a map of pairs
\begin{equation}
  \label{eq:Tprime}
  T': (E^\sharp / \ker T, E / (E \cap \ker T)) \lra (F^\sharp, F)
\end{equation}
and a simple argument shows that
$T$ satisfies~$(\mathcal{H}')$
if and only if
$T'$ satisfies~$(\mathcal{H})$.

Our goal is to investigate
closedness of the range of
$T: E \to F$
as a continuous linear map.
\begin{Lem}
  \label{lem:graph}
  Let $T: (E^\sharp, E) \to (F^\sharp, F)$
  be a map of pairs of topological vector spaces
  with $F^\sharp$ Hausdorff.
  If $T$ satisfies~$(\mathcal{H})$
  then the graph of $T: E \to F$
  is closed in $E^\sharp \times F$.
\end{Lem}
\begin{proof}
  We prove that
  the range of the continuous map
  \begin{equation*}
    \TR{\gamma_T}{u}{E}{(u, Tu)}{E^\sharp \times F}
  \end{equation*}
  is closed.
  Take $\{ u_\alpha \}$ a net in $E$
  such that
  $(u_\alpha, Tu_\alpha) \to (u, f)$ in $E^\sharp \times F$
  for some $(u, f) \in E^\sharp \times F$. Then:
  \begin{itemize}
  \item $u_\alpha \to u$ in $E^\sharp$,
    and since $T: E^\sharp \to F^\sharp$ is continuous
    we have that $Tu_\alpha \to Tu$ in $F^\sharp$;
  \item $Tu_\alpha \to f$ in $F$,
    hence also in $F^\sharp$.
  \end{itemize}
  Since $F^\sharp$ is Hausdorff
  we have that $Tu = f \in F$:
  therefore $u \in E$ thanks to property~$(\mathcal{H})$,
  hence $(u, f) = (u, Tu) \in \ran \gamma_T$.
\end{proof}

From here on
we focus on the particular situation
in which
$E, E^\sharp, F$ are DFS spaces,
and we fix
$\{E_j\}_{j \in \Z_+}$, $\{E^\sharp_k\}_{k \in \Z_+}$, $\{F_k\}_{k \in \Z_+}$
injective sequences of Banach spaces
with compact inclusion maps
whose injective limits are
$E, E^\sharp, F$ respectively.
We also assume the following condition
(stronger than $E \hookrightarrow E^\sharp$):
\begin{equation}
  \label{eq:sp_incl}
  \text{$E \hookrightarrow E^\sharp_0$ continuously}.
\end{equation}

\begin{Thm}
  \label{thm:Hclosedrange}
  Let $E, E^{\sharp}, F$ be DFS as above
  and $F^{\sharp}$ be a Hausdorff space.
  If $T$ satisfies~$(\mathcal{H})$
  then $T: E \to F$ has closed range.
\end{Thm}
\begin{proof}
  By Lemma~\ref{lem:graph}
  the graph map
  $\gamma_T$
  has closed range.
  Since both $E$ and $E^\sharp \times F$
  are DFS spaces,
  and moreover $\gamma_T$ is injective,
  the following criterion for
  closedness of its range
  applies~\cite[Lemma~2.3]{araujo18}:
  for every $k \in \Z_+$
  there exists $j \in \Z_+$
  such that
  \begin{equation*}
    \forall u \in E, \ \gamma_T(u) \in E^\sharp_k \times F_k
    \Longrightarrow u \in E_j
  \end{equation*}
  which, thanks to~\eqref{eq:sp_incl}
  and the definition of $\gamma_T$,
  is equivalent to
  \begin{equation*}
    \forall u \in E, \ Tu \in F_k
    \Longrightarrow u \in E_j
  \end{equation*}
  in turn implying that
  $T: E \to F$
  has closed range~\cite[Theorem~2.5]{araujo18}.
\end{proof}
Since the class of DFS spaces
is closed under
taking closed subspaces and quotients~\cite{komatsu67},
the whole picture above is preserved
if we replace $T$ by $T'$~\eqref{eq:Tprime}.
For instance,
$H \dfn E / (E \cap \ker T)$
is a DFS space;
actually,
$H_j \dfn E_j / (E_j \cap \ker T)$
is a Banach space,
the inclusion map
$E_j \hookrightarrow E_{j'}$ ($j < j'$)
descends to a compact injection
$H_j \hookrightarrow H_{j'}$,
and we have
\begin{equation*}
  H \cong \varinjlim H_j
\end{equation*}
where we regard $H_j$ as a subspace of $H$
by means of the identification
\begin{equation*}
  u + (E_j \cap \ker T) \in H_j
  \longmapsto
  u + (E \cap \ker T) \in H
\end{equation*}
(which the reader may check at once
to be well-defined, continuous and injective).
The same goes for $E^\sharp / \ker T$,
and~\eqref{eq:sp_incl} descends
to a continuous injection
$E / (E \cap \ker T) \hookrightarrow E^\sharp_0 / (E^\sharp_0 \cap \ker T)$
-- the first step in the sequence that
naturally defines the injective limit
topology on $E^\sharp / \ker T$.
\begin{Cor}
  \label{cor:abstract_agh_dfs}
  If $T$ satisfies~$(\mathcal{H}')$
  then $T: E \to F$ has closed range.
\end{Cor}
\begin{proof}
  In that case,
  $T'$~\eqref{eq:Tprime}
  satisfies~$(\mathcal{H})$:
  by our previous digression,
  we are entitled to apply
  Theorem~\ref{thm:Hclosedrange}
  to it,
  yielding closedness of the range
  of $T': E / (E \cap \ker T) \to F$,
  which equals that of $T: E \to F$.
\end{proof}

\subsection{Global Gevrey solvability}

Let $M$ be a compact real-analytic manifold
as in Section~\ref{sec:prelim}.
Given $s \geq 1$,
we denote by $\D'_s(M)$ the topological dual of $\gev^s(M)$,
the so-called space of Gevrey ultradistributions of order $s$
(when $s = 1$ this is the space of hyperfunctions on $M$).

Given $P$ a Gevrey LPDO on $M$,
we are interested in solving the equation
$Pu = f$ for $f \in \gev^s(M)$.
If there exists a solution $u \in \gev^s(M)$ to
this problem
then for any $v \in \D'_s(M)$ we have
\begin{equation*}
  \langle v, f \rangle
  = \langle v, Pu \rangle
  = \langle \transp{P}v, u \rangle
\end{equation*}
where $\transp{P}$ denotes
the transpose of $P$.
A necessary condition on $f$
to the solvability of
$Pu = f$ is then that
\begin{equation}
  \label{eq:cc_gevrey}
    \langle v, f \rangle = 0
    \quad \text{for every $v \in \D'_s(M)$ such that $\transp{P}v = 0$}.
  \end{equation}
This leads us to the following:
\begin{Def}
  We say that
  $P$ is \emph{globally $\gev^s$-solvable} if
  for every $f \in \gev^s(M)$
  satisfying~\eqref{eq:cc_gevrey}
  there exists $u \in \gev^s(M)$
  such that
  $Pu = f$.
\end{Def}
It turns out~\cite[Lemma~2.2]{araujo18}
that $P$ is globally $\gev^s$-solvable
if and only if the map between DFS spaces
$P: \gev^s(M) \to \gev^s(M)$
has closed range.
In~\cite[Theorem~3.5]{afr20b}
we proved that the following property,
weaker that global hypoellipticity,
implies global solvability of $P$
in the smooth setting
(i.e.~closedness of the range of
$P: \cinfty(M) \to \cinfty(M)$):
\begin{equation*}
  \forall u \in \D'(M), \ Pu \in \cinfty(M)
  \Longrightarrow
  \text{$\exists v \in \cinfty(M)$ such that $Pv = Pu$}.
\end{equation*}
Corollary~\ref{cor:abstract_agh_dfs} gives,
in the Gevrey framework,
that a weak notion of global Gevrey hypoellipticity
implies global solvability.
\begin{Thm}
  \label{thm:gevrey_agh}
  Let $s \geq 1$ and suppose that
  for some $s_+ > s$ we have
  \begin{equation}
    \label{eq:agh_gevrey}
    \forall u \in \gev^{s_+}(M), \ Pu \in \gev^s(M)
    \Longrightarrow \text{$\exists v \in \gev^s(M)$ such that $Pv = Pu$}.
  \end{equation}
  Then $P$ is globally $\gev^s$-solvable.
\end{Thm}
\begin{proof} Let
  \begin{equation*}
    E \dfn \gev^s(M),
    \quad
    E^\sharp \dfn \gev^{s_+}(M),
    \quad
    F \dfn \gev^s(M),
    \quad
    F^\sharp \dfn \D'_s(M).
  \end{equation*}
  Since with these choices~\eqref{eq:agh_gevrey}
  is per se property~$(\mathcal{H}')$
  for $T \dfn P$,
  all the hypotheses of
  Corollary~\ref{cor:abstract_agh_dfs}
  are automatically fulfilled.
\end{proof}

The converse of
Theorem~\ref{thm:gevrey_agh}
holds for operators in the class~$\mathcal{T}$:
\begin{Prop}
  \label{prop:converse_agh_classT}
  Let $P$ be a real-analytic LPDO
  on $T \times G$ in class $\mathcal{T}$.
  If $P$ is
  globally $\gev^s$-solvable
  then
  \begin{equation*}
    \forall u \in \D'(T \times G), \ Pu \in \gev^s(T \times G)
    \Longrightarrow \text{$\exists v \in \gev^s(T \times G)$ such that $Pv = Pu$}.
  \end{equation*}
\end{Prop}
\begin{proof}
  Let $u \in \D'(T \times G)$ be such that
  $f \dfn Pu \in \gev^s(T \times G)$.
  Thanks to
  Proposition~\ref{Pro:partial-convergence-Gs}
  we have
  \begin{equation}
    \label{eq:solv_implies_agh_gevrey}
    f
    = \sum_{\lambda \in \sigma(\Delta_G)} \mathcal{F}_\lambda^G(f)
    = \lim_{\nu \to \infty} \sum_{|\lambda| \leq \nu} \widehat{P}_\lambda \mathcal{F}_\lambda^G(u)
    = \lim_{\nu \to \infty} P \sum_{|\lambda| \leq \nu} \mathcal{F}_\lambda^G(u)
  \end{equation}
  with convergence in
  $\gev^s(T \times G)$.
  Since each
  $\widehat{P}_\lambda$ is elliptic
  we have that
  $\widehat{P}_\lambda \mathcal{F}_\lambda^G(u) = \mathcal{F}_\lambda^G(f) \in \gev^s(T; E_\lambda^G)$
  implies that
  $\mathcal{F}_\lambda^G(u) \in \gev^s(T; E_\lambda^G)$,
  hence~\eqref{eq:solv_implies_agh_gevrey}
  ensures that $f$
  belongs to the closure of the range of
  $P: \gev^s(T \times G) \to \gev^s(T \times G)$.
  As the latter
  is closed in $\gev^s(T \times G)$
  by assumption,
  there exists
  $v \in \gev^s(T \times G)$ such that
  $Pv = f = Pu$.
\end{proof}

\bibliographystyle{plain}
\bibliography{bibliography}

\begin{thebibliography}{10}

\bibitem{albanese11}
A.~A. Albanese.
\newblock On the global {$C^\infty$} and {G}evrey hypoellipticity on the torus
  of some classes of degenerate elliptic operators.
\newblock {\em Note Mat.}, 31(1):1--13, 2011.

\bibitem{aj14}
A.~A. Albanese and D.~Jornet.
\newblock Global regularity in ultradifferentiable classes.
\newblock {\em Ann. Mat. Pura Appl. (4)}, 193(2):369--387, 2014.

\bibitem{araujo18}
G.~Ara\'ujo.
\newblock Regularity and solvability of linear differential operators in
  {G}evrey spaces.
\newblock {\em Math. Nachr.}, 291(5-6):729--758, 2018.

\bibitem{afr20}
G.~{Ara{\'u}jo}, I.~A. Ferra, and L.~F. Ragognette.
\newblock {Global hypoellipticity of sums of squares on compact manifolds}.
\newblock {\em arXiv e-prints}, page arXiv:2005.04484, May 2020.

\bibitem{afr20b}
G.~{Ara{\'u}jo}, I.~A. {Ferra}, and L.~F. {Ragognette}.
\newblock Global solvability and propagation of regularity of sums of squares
  on compact manifolds.
\newblock {\em J. Anal. Math.}, (to appear), 2022.

\bibitem{bmz12}
A.~P. Bergamasco, G.~A. Mendoza, and S.~Zani.
\newblock On global hypoellipticity.
\newblock {\em Comm. Partial Differential Equations}, 37(9):1517--1527, 2012.

\bibitem{bcm79}
P.~Bolley, J.~Camus, and C.~Mattera.
\newblock Analyticit\'e microlocale et it\'er\'es d'op\'erateurs.
\newblock In {\em S\'eminaire {G}oulaouic-{S}chwartz (1978/1979)}, pages Exp.
  No. 13, 9. \'Ecole Polytech., Palaiseau, 1979.

\bibitem{brccj16}
N.~Braun~Rodrigues, G.~Chinni, P.~D. Cordaro, and M.~R. Jahnke.
\newblock Lower order perturbation and global analytic vectors for a class of
  globally analytic hypoelliptic operators.
\newblock {\em Proc. Amer. Math. Soc.}, 144(12):5159--5170, 2016.

\bibitem{chavel_eigenvalues}
I.~Chavel.
\newblock {\em Eigenvalues in {R}iemannian geometry}, volume 115 of {\em Pure
  and Applied Mathematics}.
\newblock Academic Press, Inc., Orlando, FL, 1984.
\newblock Including a chapter by Burton Randol, With an appendix by Jozef
  Dodziuk.

\bibitem{christ94}
M.~Christ.
\newblock Global analytic hypoellipticity in the presence of symmetry.
\newblock {\em Math. Res. Lett.}, 1(5):559--563, 1994.

\bibitem{chim94}
P.~D. Cordaro and A.~A. Himonas.
\newblock Global analytic hypoellipticity of a class of degenerate elliptic
  operators on the torus.
\newblock {\em Math. Res. Lett.}, 1(4):501--510, 1994.

\bibitem{grauert58}
H.~Grauert.
\newblock On {L}evi's problem and the imbedding of real-analytic manifolds.
\newblock {\em Ann. of Math. (2)}, 68:460--472, 1958.

\bibitem{hp00}
A.~A. Himonas and G.~Petronilho.
\newblock Global hypoellipticity and simultaneous approximability.
\newblock {\em J. Funct. Anal.}, 170(2):356--365, 2000.

\bibitem{hp06}
A.~A. Himonas and G.~Petronilho.
\newblock On {$C^\infty$} and {G}evrey regularity of sublaplacians.
\newblock {\em Trans. Amer. Math. Soc.}, 358(11):4809--4820, 2006.

\bibitem{hormander71b}
L.~H\"{o}rmander.
\newblock Uniqueness theorems and wave front sets for solutions of linear
  differential equations with analytic coefficients.
\newblock {\em Comm. Pure Appl. Math.}, 24:671--704, 1971.

\bibitem{knapp_lgbi}
A.~W. Knapp.
\newblock {\em Lie groups beyond an introduction}, volume 140 of {\em Progress
  in Mathematics}.
\newblock Birkh\"auser Boston, Inc., Boston, MA, 1996.

\bibitem{komatsu67}
H.~Komatsu.
\newblock Projective and injective limits of weakly compact sequences of
  locally convex spaces.
\newblock {\em J. Math. Soc. Japan}, 19:366--383, 1967.

\bibitem{kn62}
T.~Kotake and M.~S. Narasimhan.
\newblock Regularity theorems for fractional powers of a linear elliptic
  operator.
\newblock {\em Bull. Soc. Math. France}, 90:449--471, 1962.

\bibitem{kothe_tvs2}
G.~K{\"o}the.
\newblock {\em Topological vector spaces. {II}}, volume 237 of {\em Grundlehren
  der Mathematischen Wissenschaften [Fundamental Principles of Mathematical
  Science]}.
\newblock Springer-Verlag, New York-Berlin, 1979.

\bibitem{petronilho09}
G.~Petronilho.
\newblock On {G}evrey solvability and regularity.
\newblock {\em Math. Nachr.}, 282(3):470--481, 2009.

\bibitem{rodino_gevrey}
L.~Rodino.
\newblock {\em Linear partial differential operators in {G}evrey spaces}.
\newblock World Scientific Publishing Co., Inc., River Edge, NJ, 1993.

\bibitem{varadarajan84}
V.~S. Varadarajan.
\newblock {\em Lie groups, {L}ie algebras, and their representations}, volume
  102 of {\em Graduate Texts in Mathematics}.
\newblock Springer-Verlag, New York, 1984.
\newblock Reprint of the 1974 edition.

\bibitem{wallach_haohs}
N.~R. Wallach.
\newblock {\em Harmonic analysis on homogeneous spaces}.
\newblock Marcel Dekker, Inc., New York, 1973.
\newblock Pure and Applied Mathematics, No. 19.

\end{thebibliography}
\end{document}